\documentclass[preprint,3p,authoryear]{elsarticle}
\journal{Studia Scientiarum Mathematicarum Hungarica}
\usepackage{amsthm,amssymb,amsmath}
\usepackage{enumitem}
\usepackage{algorithm}
\usepackage{algorithmic}
\usepackage{multirow}
\usepackage{xcolor}

\usepackage{url}

\newtheorem{theorem}{Theorem}[section]
\newtheorem{definition}{Definition}[section]
\newtheorem{proposition}{Proposition}[section]
\newtheorem{corollary}{Corollary}[section]
\newtheorem{lemma}{Lemma}[section]
\newtheorem{example}{Example}[section]
\newtheorem{problem}{Problem}[section]

\renewcommand{\Pr}{\mathbb{P}}
\newcommand{\E}{\mathbb{E}}

\begin{document}

\begin{frontmatter}
\title{A characterization of signed discrete infinitely divisible distributions \footnote{Accepted for publication in
Studia Scientiarum Mathematicarum, 10 October 2016.}}
\author[label1]{Huiming Zhang }
\ead{zhanghuiming@pku.edu.cn}
\address[label1]{School of Mathematical Sciences, Peking University, Beijing, 100871, China}
\address[label2]{Department of Mathematics and Statistics, Central China Normal University, Wuhan, 430079, China}
\author[label2]{Bo Li}
\ead{haoyoulibo@163.com}
\author[label3]{G. Jay Kerns}
\ead{gkerns@ysu.edu}
\address[label3]{Department of Mathematics and Statistics, Youngstown State University, Youngstown, Ohio 44555-0002 USA}

\begin{abstract}  In this article, we give some reviews concerning negative probabilities model and quasi-infinitely divisible at the beginning. We next extend Feller's characterization of discrete infinitely divisible distributions to signed discrete infinitely divisible distributions, which are discrete pseudo compound Poisson (DPCP) distributions with connections to the L{\'e}vy-Wiener theorem. This is a special case of an open problem which is proposed by \cite{sato14}, Chaumont and Yor (2012). An analogous result involving characteristic functions is shown for signed integer-valued  infinitely divisible distributions. We show that many distributions are DPCP by the non-zero p.g.f.~property, such as the mixed Poisson distribution and fractional Poisson process. DPCP has some bizarre properties, and one is that the parameter $\lambda $ in the DPCP class cannot be arbitrarily small.
\end{abstract}
\begin{keyword}
discrete distribution \sep quasi infinitely divisible \sep pseudo compound Poisson \sep signed measure \sep negative  probability \sep absolutely convergent Fourier series \sep J{\o}rgensen set.

MSC2010: 60E07 60E10 28A20 42A32
\end{keyword}

\end{frontmatter}

\section{Convolutions of signed measure model and quasi-infinitely divisible}

\cite{szekely05} spoke of flipping two ``half-coins'' (which have
infinitely many sides numbered $0, 1, 2, \ldots$ whose even values are
assigned negative probabilities) to obtain a fair coin with outcomes 0
or 1 with probability $1/2$ each. The negative probabilities arise
because his probability generating function (p.g.f.)  $G(z) =
\sqrt{0.5 + 0.5z} $ has negative coefficients. He went on to
consider the general $n$-th root of a p.g.f.~as a generating function
with negative coefficients. In this work we continue along the same
lines. In short, the aim of this paper is to determine necessary and
sufficient conditions on a discrete distribution
such that the ${[G(z)]^{\frac{1}{n}}}$ (or ${[\varphi (\theta
  )]^{\frac{1}{n}}}$) is also the p.g.f.~(or characteristic function)
of a signed measure with bounded total variation. Sz\'{e}kely used the word ``signed infinitely divisible'' to describe a phenomenon that writing a discrete probability mass as a convolution of signed point measures.

Notice that central to the inversion problem of the Central Limit
Theorem is the search for some characteristic function $\varphi
(\theta)$ such that ${[\varphi (\theta )]^{\frac{1}{n}}}$ is also a
characteristic function. This is indeed the very definition of an
infinitely divisible distribution, or the deconvolution problem. In the foundation of this
body of work is the fundamental L{\'e}vy-Khinchine result: ``A
distribution is infinitely divisible if and only if it has a
L{\'e}vy-Khinchine representation.''

Continuing along these lines, \cite{sato14}, \cite{nakamura13}
proposed the following definition of quasi infinite divisibility for distributions having a
L{\'e}vy-Khinchine representation but with \emph{signed} L{\'e}vy
measure. By constructing a complete Riemann zeta distribution
corresponding to Riemann hypothesis, \cite{nakamura15} showed that a
complete Riemann zeta distribution is quasi-infinitely divisible for
some conditions. Based on quantum physics, \cite{demni15} constructed a generalized Poisson distribution and derived a L{\'e}vy-Khintchine-type representation of its characteristic function with signed L{\'e}vy measure.

A random variable $X$ on $\Bbb{R}$ is called
\emph{quasi-infinitely divisible} if the characteristic function of
$X$ has the following form.

\begin{definition}[Quasi-infinitely divisible] \label{def:quasi-ID} A
  distribution $\mu$ on $\Bbb{R}$ is said to be quasi-infinitely
  divisible if its characteristic function has the form
  \begin{equation} \label{eq:levy-khinchine}
    \E \mathrm{e}^{\mathrm{i}\theta X} = \exp \left\{
      a\mathrm{i}\theta - \frac{1}{2} \sigma^{2} \theta^{2} + \int_{\Bbb{R}}
      \left( \mathrm{e}^{\mathrm{i}\theta x} - 1 - \frac{\mathrm{i}\theta x}{1 +
          x^{2}} \right) \, \nu (\mathrm{d}x) \right\},
  \end{equation}
with $a \in \Bbb{R}$, $\sigma \geqslant 0$, and corresponding measure $\nu$
a bounded signed measure (that is, a \emph{quasi-L{\'e}vy measure}) on $\Bbb{R}$
with total variation measure $\left| \nu \right|$ satisfying
$\smallint_{\Bbb{R}}\min ({x^2},1)\left| \nu \right|(\mathrm{d}x) < \infty
$.
\end{definition}
The above definition also appears in Exercise 12.3 of \cite{sato13},
p.~66, where it is shown that $X$ is not infinitely divisible (in the
classical sense) if $\nu$ is a signed measure. The L{\'e}vy-Khinchine
representation with signed L{\'e}vy measure is unique; see
Exercise 12.3 in \cite{sato13}.

\begin{problem}
  Find a necessary and sufficient condition for the L{\'e}vy-Khinchine
  representation with signed L{\'e}vy measure to hold.
\end{problem}

This is an open problem posed by Professor Ken-iti Sato, see p29 of
\cite{sato14}. When $X$ in (\ref{eq:levy-khinchine}) is non-negative
(a ``subordinator'' version of (\ref{eq:levy-khinchine})), it is given
as an unsolved problem in Exercise 4.15(6) of \cite{chaumont12}.

Denote the nonnegative integers by $ \Bbb N= \{ 0,1,2, \ldots \} $. Let $\nu$ be a signed point measure on $\Bbb N$ and remove the normal
component in \eqref{eq:levy-khinchine}. Then the characteristic
function of the DPCP distribution (see Definition 2.1 below) is a discrete version of the
L{\'e}vy-Khinchine representation with signed L{\'e}vy measure.

\cite{baishanski98} considered a complex-valued (which includes
negative-valued) probability model for $n$-fold convolutions of
i.i.d.~integer-valued ``random variables'' ($X_{1}, \ldots , X_{n}$)
with complex-valued probabilities $\Pr(X_{1} = v) = a_{v}$. The
characteristic function is
$\varphi (\theta ) = \sum a_{v} \mathrm{e}^{\mathrm{i}v\theta }$, with
$\sum a_{v} = 1$, and $\Pr(X_{1} + \cdots + X_{n} = v) = a_{nv}$. He
charted this territory perhaps not for the sake of statisticians or
probabilists, but certainly to the benefit of analysts. His work stemmed from an open problem related to the complex-valued
probability model which was first posed by \cite{beurling38} and later
quoted by \cite{helson53}. It is called the problem of ``Norms of powers
of absolutely convergent Fourier series'' and has been investigated at
length in many, many papers; see \cite{baishanski98} for a review. We
only present the problem here and show its relationship to the DPCP
distribution.
\begin{problem}
  Let $f(\theta ) = \sum\limits_{j = 0}^\infty a_{j} \mathrm{e}^{\mathrm{i}j\theta }
  \mbox{ and } \left\| f \right\| = \sum \limits_{j = 0}^{\infty} \left|
      a_{j} \right| < \infty $.  Under what conditions on $f$ are
  the $\left\| f^{n} \right\|$ bounded?  Discuss the asymptotic
  behavior of $\left\| f^{n} \right\|$ as $n \to \infty$.
\end{problem}

When $a_{v} \geqslant 0$, the behavior of $a_{nv}$ has been firmly
established via the Central Limit Theorem and gives rise to the normal
law in particular, and stable laws in general. In the case of
complex-valued probability, \cite{baishanski98} asks,
\begin{problem}
What is the Central Limit Theorem for complex-valued probabilities?
\end{problem}
\noindent For some results in this direction, see \cite{baez93}
and \cite{baishanski98}.

Here, we consider the case of real-valued probabilities. \cite{van81} proved Kolmogorov's Extension Theorem for finite signed
measures which guarantees that a suitably ``consistent'' collection of
finite-dimensional distributions will define a signed
r.v..~\cite{kemp79} studied the circumstances under which convolutions
of binomial variables and binomial pseudo-variables (with negative
probability density) lead to valid distributions (with positive
probability density). \cite{karymov05} obtained asymptotic
decompositions into convolutions of Poisson signed measure that is appropriate for a broad range
of lattice distributions.

Let $(\Omega ,{\cal F},\mu )$ be a
signed measure space with $\mu ({\Omega}) = 1$ and $(S, {\cal B})$ be
a measurable space. A mapping $X:\Omega \mapsto \Bbb{R}$ is a
\emph{signed random variable} if it is a measurable function from
$(\Omega ,{\cal F})$ to $(S,{\cal B})$.
Every signed random variable $X$ has an associated signed measure. The signed measure is $\mu (B) = \Pr(X \in B) = \Pr({X^{ - 1}}B)$ for
each $B \in {\cal B}$.

\textbf{Remark 1} : In this paper, we write ``r.v.''(or ``distribution'') for a random
variable with ordinary (not negative) probabilities, and we write
``signed r.v.'''(or ``signed distribution'') for a random variable that permits negative
probabilities.

\begin{definition}[Signed probability density] A signed random
  variable $X$ has signed probability distribution
  $\mu (\mathrm{d}x) \in \Bbb{R}$ satisfying
  ${\smallint _{\Omega}}\,\mu (\mathrm{d}x) = 1 \mbox{ with }
  \smallint_{\Omega} \vert \mu \vert (\mathrm{d}x) < \infty $.
\end{definition}

The signed discrete distribution is well-defined provided the total
variation ${\smallint _{\Omega}}\vert\mu\vert(\mathrm{d}x)$ is
finite. Taking a signed discrete distribution as an example,
the absolute convergence of $\sum\limits_{k = 1}^{\infty} a_{k}$
guarantees that all rearrangements of the series are convergent to the
same value. Signed random variables defined this way allow for treatment
analogous to the classical case with the concepts of independence,
expectation, variance, $r^{\mathrm{th}}$ moments, characteristic
functions, \emph{etc}.~operating in the natural way.

Without the condition of absolute convergence the negativity of
$\alpha_{k}$ would cause $\exp \left\{\sum\limits_{i = 1}^{\infty}
  \alpha_{i}\lambda (z^{i} - 1) \right\}$ to be undefinable, since we
know from the Riemann series theorem that a conditionally convergent
series can be rearranged to converge to any desired value.  The next
example drives home this point.

\begin{example}[Discrete uniform distribution]
  Let $X$ be Bernoulli r.v. with probability of success $p=0.5$. The
  logarithm of the p.g.f.~of $X$ is
  \[\ln \left(\frac{1}{2} + \frac{1}{2}z \right) = \ln \frac{1}{2} +
  \sum\limits_{i = 1}^{\infty} \frac{(-1)^{i - 1}}{i} z^{i} = \ln
  2\sum\limits_{i = 1}^{\infty} \frac{(-1)^{i - 1}}{i\ln 2} (z^{i} -
  1).\]
  However, $\sum\limits_{i = 1}^{\infty} \left| a_{i} \right| =
  \sum\limits_{i = 1}^{\infty} \frac{1}{i\ln 2} = \infty$.
\end{example}

The present paper has concerned itself with the study of certain
classes of discrete random variables, in particular, those of the
infinitely divisible variety. We started from William Feller's famous
characterization that all discrete infinitely divisible distributions
are compound Poisson distributed. Now, not all discrete distributions
are infinitely divisible (not, at least, in the classical sense). But
following the idea of Sz\'{e}kely and others we extend our notion of
infinitely divisible to include those distributions whose ${n^{th}}$
root of their p.g.f.~is also a p.g.f.~in a generalized sense. Szekely's ``signed'' ID is based on convolution of signed measures and Sato defines his ``quasi'' ID from L{\'e}vy-Khinchine representation with signed L{\'e}vy measure.
The goal of this paper is to find a connection with these two kinds of generalised ID under the discrete r.v..

The paper is structured as follows: In Section 2, after giving the definition of signed probability model, we present several conditions guaranteeing that a discrete distribution is a signed discrete infinitely divisible distribution (or discrete pseudo compound Poisson). Also, we exemplify some famous discrete distributions which belong to the discrete quasi infinitely divisible distributions, such as the mixed Poisson distribution and the fractional Poisson process. In the same way, in Section 3, we conclude
that a distribution is signed integer-valued ID if and only if it is
integer-valued pseudo compound Poisson. In Section 4, some bizarre properties of signed discrete infinitely divisible distributions are discussed, and we mention a research problem of finding characteristic function's J{\o}rgensen set.

\section{Feller's characterization and extension of signed discrete ID}
\subsection{Discrete pseudo compound Poisson distribution}
Feller's characterization of the compound Poisson distribution states
that a non-negative integer valued r.v.~$X$ is infinitely divisible
(ID) if and only if its distribution is a discrete compound Poisson
distribution. Taking $N$ and $Y_{i}$'s to be independent, $X$ can be
written as
\begin{equation} \label{eq:ID}
X = Y_{1} + Y_{2} +  \cdots  + Y_{N},
\end{equation}
where $N$ is Poisson distributed with parameter $\lambda $ and the
$Y_{i}$'s are independently and identically distributed (i.i.d.)
discrete r.v.'s with $P\{ {Y_1} = k\} = \alpha_{k}$. Hence the
p.g.f.~of the compound Poisson distribution can be written as
\begin{equation}  \label{eq:feller-CP-pgf}
  P(z) = \sum\limits_{k = 0}^{\infty}  P_{k} z^{k} = \exp\left\{ \lambda \sum\limits_{k = 1}^{\infty}  \alpha_{k} (z^{k} - 1) \right\}, \quad \vert z \vert \leqslant 1,\ \lambda  > 0,\ \sum\limits_{k = 1}^{\infty}  \alpha_{k} = 1,\ \alpha_{k} \geqslant 0,
\end{equation}
where $z$ is a real number. (In this paper, all the arguments $z$ of a
p.g.f.~are taken to be real numbers $z \in [-1,1]$.) For more properties and characterizations of discrete compound Poisson
, see \cite{janossy50}, Section 12.2 of \cite{feller68}, Section 9.3 of
\cite{johnson05}, and \cite{zhang16}.

However, Feller neither shows nor claims that the sum of the
coefficients in $ n\left(\root n \of {P(z)} - 1\right)$ is bounded
for any $n$, that is, it leaves the question of
whether
\[\ln P(z) = \mathop {\lim }\limits_{n \to \infty }
\sum\limits_{k = 1}^\infty {n{a_{nk}}} \left({z^k} -
  1\right)\hspace{0.15in} \mathop = \limits^?
\hspace{0.15in}\sum\limits_{k = 1}^\infty b_{k} (z^{k} - 1), \quad
\mbox{where }\sum\limits_{k = 1}^\infty {{b_k} < + \infty }.\]
This result of Feller's can be viewed as a discrete analogue of the derivation
of L{\'e}vy-Khinchine's formula; see \cite{ito05}, \cite{sato13}, \cite{meerschaert01}
for the general case. When some $\alpha_{nk}$ are negative, it is necessary to find a new method which guarantees that
the extension of Feller's characterization relating to the
$n^{\mathrm{th}}$ convolution of a signed measure is valid.

If it turns out that some $\alpha_{i}$ are negative in the p.g.f.~of
\eqref{eq:feller-CP-pgf}, then the $Y_{i}$ in \eqref{eq:ID} will have
negative probability and \emph{a fortiori} will rule out any chance
for $\varphi (z) = \exp \left\{\sum\limits_{k = 1}^{\infty}
  \alpha_{k}\lambda (z^{k} - 1) \right\}$ to be an infinitely
divisible p.d.f. For instance:
\begin{example}
From the Taylor series expansion it follows that
\[\frac{2}{3} + \frac{1}{3}z = \exp \left\{\ln \tfrac{2}{3} + \sum\limits_{k = 1}^\infty  \frac{(-1)^{k - 1}}{k 2^{k}} z^{k} \right\} = \exp\left\{\ln \tfrac{3}{2}\sum\limits_{k = 1}^\infty  \frac{(-1)^{k - 1}}{k2^{k}\ln \tfrac{3}{2}}(z^{k} - 1) \right\}.\]
Note that $\alpha_{k} = \frac{(-1)^{k - 1}}{k2^{k} \ln
    \tfrac{3}{2}}$ is negative whenever $k$ is even.
\end{example}

\begin{definition}[Discrete pseudo compound Poisson distribution, DPCP] \label{def:DPCP}
  If a discrete r.v.~$X$ with $\Pr(X = i) = P_{i}$, $i \in \Bbb N$, has a
  p.g.f.~of the form
  \begin{equation} \label{eq:DPCP}
    P(z) = \sum\limits_{i = 0}^\infty  P_{i} z^{i} = \exp \left\{\sum\limits_{k = 1}^{\infty} \alpha_{k} \lambda (z^{k} - 1) \right\},
  \end{equation}
  where $\sum\limits_{i = 1}^{\infty} \alpha_{i} = 1$, $\sum \limits_{i =
    1}^{\infty} \left| \alpha_{i} \right| < \infty$, $\alpha_{i} \in
  \Bbb{R}$, and $\lambda > 0$, then $X$ is said to have a
  \emph{discrete pseudo compound Poisson} distribution, abbreviated
  \emph{DPCP}. We denote it as $X \sim \emph{DPCP}({\alpha _1}\lambda,{\alpha _2}\lambda, \cdots )$. The $r$-parameter case: If $X \sim \emph{DPCP}({\alpha _1}\lambda,{\alpha _2}\lambda, \cdots ,{\alpha _r}\lambda)$, we say $X$ has a \emph{DPCP} distribution of order $r$.
\end{definition}

Here is an explicit expression for the probability mass function of
the DPCP distribution: \[P_{0} = \mathrm{e}^{-\lambda},\ P_{1} =
\alpha_{1} \lambda \mathrm{e}^{ - \lambda },\ P_{2} = \left(
  \alpha_{2} \lambda + \frac{1}{2!} \alpha_{1}^{2} \lambda^{2} \right)
\mathrm{e}^{-\lambda},\ P_{3} = \left( \alpha_{3} \lambda + \alpha_{1}
  \alpha_{2} \lambda^{2} + \frac{1}{3!} \alpha_{1}^{3} \lambda^{3}
\right) \mathrm{e}^{ -\lambda }, \ldots \]
\begin{equation} \label{eq:DPCP-explicit}
  P_{n} = \left({\alpha _n}\lambda  +  \cdots  + \sum\limits_{\scriptstyle{k_1} +  \cdots  + {k_u} +  \cdots {k_n} = i,{k_u} \in {\Bbb N }\hfill\atop
    \scriptstyle1\cdot{k_1} +  \cdots  + u{k_u} +  \cdots  + n{k_n} = n\hfill} {\frac{{\alpha _1^{{k_1}}\alpha _2^{{k_2}} \cdots \alpha _n^{{k_n}}{\lambda ^i}}}{{{k_1}!{k_2}! \cdots {k_n}!}}}  +  \cdots  + \frac{{\alpha _1^n{\lambda ^n}}}{{n!}} \right) {\mathrm{e}^{ -\lambda }},
\end{equation}
see \cite{janossy50}, \cite{johnson05} for the discrete compound Poisson case (all $\alpha_{i}$ are non-negative).

\subsection{Characterizations}
It turns out that there already exist a few characterizations of the
DPCP distribution in the literature. Indeed, \cite{levy37} derived the
recurrence relation
\begin{equation}  \label{eq:recursive}
P_{j + 1} = \frac{\lambda}{j + 1} \left[\alpha_{1} P_{j} + 2\alpha_{2} P_{j - 1} +  \cdots  + (j + 1) \alpha_{j + 1} P_{0} \right], \quad {P_0} = \mathrm{e}^{-\lambda},\ j = 1,2, \ldots .
\end{equation}
 of the density of the DPCP distribution when
$\alpha_{i}$ might be negative-valued. If we take $P_{j}$ to be the
empirical probability mass function, then the recursive formula in \eqref{eq:recursive}
can be used to estimate the parameters $\alpha_{j}$, for $j = 1,2,
\ldots$ (see \cite{buchmann03}, p.~1059).

The name ``pseudo compound Poisson'' was introduced by
\cite{hurlimann90}.  For the
general situation, the following L{\'e}vy-Wiener theorem provides us a
shortcut on necessary and sufficient conditions for a distribution to
be DPCP. The proof is non-trivial; see \cite{zygmund02} or
\cite{levy35}. The simple case $H(t) = {t^{ - 1}}$ is due to Wiener.

\begin{lemma}[L{\'e}vy-Wiener theorem] \label{thm:levy-wiener} Let
  $F(\theta) = \sum\limits_{k = -\infty }^{\infty} c_{k}
  \mathrm{e}^{\mathrm{i}k\theta },$ for $\theta \in [0,2\pi]$, be an
  absolutely convergent Fourier series with $\left\| F \right\| =
  \sum\limits_{k = -\infty }^{\infty} \left| c_{k} \right| <
  \infty$. The values of $F(\theta)$ lie on a curve $C$. Let $H(t)$ be
  a holomorphic function (analytic function) of a complex variable $t$
  which is regular at every point of $C$. Then $H[F(\theta)]$ has an
  absolutely convergent Fourier series.
\end{lemma}

The next two corollaries are direct consequences of the
L{\'e}vy-Wiener theorem.

\begin{corollary} \label{cor:levy-wiener1}
  For each $k \in \mathbb{N}$, the $\frac{1}{k}^{\mathrm{th}}$ power of the
  p.g.f.~$G(z) = \sum\limits_{i = 0}^{\infty} p_{i}z^{i}$ is
  \[\sqrt[k]{G(z)} = \sum\limits_{i = 0}^{\infty} q_{i}^{(k)}z^{i},\quad \left| z \right| \le 1.\]
  If $G(z)$ has no zero, then $\sqrt[k]{{G(z)}}$ is absolutely
  convergent, namely, $\sum\limits_{i = 0}^{\infty} \left| q_{i}^{(k)}
  \right| < \infty$.
\end{corollary}

When $p_{0} \ge p_{1} \ge p_{2} \ge \ldots \ge 0$ as in Corollary
\ref{cor:levy-wiener1}, this is Sz\'{e}kely's Discrete Convex Theorem
(see Theorem 2.3.1 of \cite{kerns04}, p.~29).

\begin{corollary} \label{cor:levy-wiener2} Let $f(\theta ) =
  \sum\limits_{j = 0}^{\infty} a_{j} \mathrm{e}^{\mathrm{i}j\theta}$
  and $\left\| f \right\| = \sum\limits_{j = 0}^{\infty} \left| a_{j}
  \right| < \infty $. If $f(\theta )$ has no zero, then the $\left\| f^{n}
  \right\|$ are bounded.
\end{corollary}

Next, we give a lemma about the non-vanishing p.g.f. characterization of DPCP, see \cite{zhang2014}. We restate the proof here.

\begin{lemma}[Non-zero p.g.f. of DPCP] \label{lem:char-DPCP}
  For any discrete r.v.~$X$, its p.g.f.~$G(z)$ has no zeros if and
  only if $X$ is DPCP distributed.
\end{lemma}
\begin{proof}
  It is easy to see that the p.g.f.~of a DPCP distribution has no
  zero. On the other hand, if $G(z)$ has no zeros, taking $z$ as a
  complex number, let $z = r \mathrm{e}^{\mathrm{i}\theta}$, for $1
  \geqslant r \geqslant 0$. We have $G(r \mathrm{e}^{\mathrm{i}\theta
  }) = \sum\limits_{k = 0}^{\infty} p_{k} r^{k}
  \mathrm{e}^{\mathrm{i}k\theta}$, thus $\sum\limits_{k = 0}^{\infty}
  \left| p_{k} r^{k} \right| \leqslant \sum\limits_{k = 0}^{\infty}
  \left| p_{k} \right| = 1$. By applying the L{\'e}vy-Wiener theorem
  for all $r \in [0,1]$, $\ln G(\mathrm{e}^{\mathrm{i}\theta})$ has
  an absolutely convergent Fourier series. Therefore, $X$ is DPCP
  distributed from Definition \ref{def:DPCP}.
\end{proof}

For instance,
$P(z) = \frac{1}{3} + \frac{2}{3}z$ on $\left| z \right| \leqslant 1$
has no pseudo compound Poisson representation since
$\frac{1}{3} + \frac{2}{3}z = 0$ for $z = -\frac{1}{2}$.

Next, we define signed discrete infinite divisibility (ID) as an extension
of discrete infinite divisibility. Firstly, we show that p.g.f. of a signed discrete infinitely divisible distribution never vanishes. Secondly, we obtain an
extension of Feller's characterization by employing the
L{\'e}vy-Wiener theorem.

\begin{definition} \label{def:quasi-ID}
  A p.g.f.~is said to be \emph{signed discrete infinitely divisible} if for
  every $n \in \mathbb{N}$, $G(z)$ is the $n^{\mathrm{th}}$-power of some
  p.g.f.~with signed probability density, namely
  \[G(z) = [G_{n}(z)]^{n} = \left[ \sum \limits_{i = 0}^{\infty}
    p_{i}^{(n)}z^{i} \right]^{n}, \quad \sum\limits_{i = 0}^{\infty}
  p_{i}^{(n)} = 1,\ \sum\limits_{i = 0}^{\infty} \left| p_{i}^{(n)}
  \right| < \infty, \ p_i^{(n)} \in \Bbb{R}.\]
\end{definition}

The notion of signed discrete infinite divisibility first appeared in
\cite{szekely05} where he discusses the conditions under which
$\sqrt[k]{{G(z)}}$ is absolutely convergent in the special case that
$G(z)$ is the p.g.f.~of a Bernoulli distribution.

To get a characterization for signed discrete ID distributions, we need
Prohorov's theorem for signed measures, see p.~202 of
\cite{bogachev07}. Applying Prohorov's theorem for bounded and
uniformly tight signed measures generalises the continuity theorem for
signed p.g.f.'s.

\begin{lemma}[Prohorov's theorem for signed measures,
  \cite{bogachev07}] \label{lem:prohorov} Let $(E,\tau )$ be a
  complete separable metric space and let ${\cal M}$ be a family of
  signed Borel measures on $E$. Then the following conditions are
  equivalent:
  \begin{enumerate}[label={\upshape(\roman*)}]
  \item Every sequence $\mu_{n} \in {\cal M}$ contains a weakly
    convergent subsequence.
  \item The family ${\cal M}$ is uniformly tight and bounded in the
    variation norm (a signed measure ${\mu}$ in a topological space
    $(E,\tau)$ is called \emph{uniformly tight} if for every
    $\varepsilon > 0$ there exists a compact set $K_{\varepsilon}$
    such that $\left| \mu \right|(E\backslash {K_\varepsilon }) <
    \varepsilon $ for all $\mu \in {\cal M}$).
  \end{enumerate}
\end{lemma}
Let $G$ be as in Definition \ref{def:quasi-ID} and take $E=\Bbb N$ and
${K_{\varepsilon M}} = \{ 0,1,\ldots, M - 1, M\}$ in Lemma
\ref{lem:prohorov}. For every $n \in \mathbb{N}$, $\sum\limits_{i = 0}^{\infty} \left|
  p_{i}^{(n)} \right| < \infty$, there exists a compact set
$K_{\varepsilon M}$ such that $\sum\limits_{i = M+1}^{\infty} \left|
  p_{i}^{(n)} \right| < \varepsilon $ for every $\varepsilon > 0$, namely $\mathop {\lim }\limits_{M \to \infty } \mathop {\sup }\limits_n \sum\limits_{i = M + 1}^\infty  {\left| {p_i^{(n)}} \right|}  = 0$ (see p.~42 in \cite{sato13} for the case of positive measure). So
$\{p_{i}^{(n)}\}_{i = 0}^{\infty}$ is a sequence of uniformly tight
signed bounded point measures.

The next lemma is an extension of the continuity theorem for
p.g.f.'s. With slight modifications, the necessity part directly
follows the proof of the p.g.f.~case; see \cite{feller68} p.~280.

\begin{lemma}[Continuity theorem for signed p.g.f.'s]
  Let ${G_n}(z) = \sum\limits_{k = 0}^\infty {p_k^{(n)}} {z^k}$ on $\left|
    z \right| \le 1$ be p.g.f.'s for a sequence of bounded signed
  point measures $\left\{ p_k^{(n)} \right\}_{k = 0}^{\infty}$ with
  $\sum\limits_{k = 0}^\infty {\left| {p_k^{(n)}} \right|} < \infty
  $. Then there exists a sequence $p_{k}$ such that
  \begin{equation} \label{eq:cont-thm-signed1}
    \mathop {\lim }\limits_{n \to \infty } p_k^{(n)} = {p_k},
  \end{equation}
  for $k = 0,1, \ldots$, if and only if the limit
  \begin{equation} \label{eq:cont-thm-signed2}
    \mathop {\lim }\limits_{n \to \infty } {G_n}(z) = \mathop {\lim }\limits_{n \to \infty } \sum\limits_{k = 0}^\infty  {p_k^{(n)}} {z^k} \buildrel \Delta \over = G(z)
  \end{equation}
  exists for each $z$ in the open interval $(0,1)$. Furthermore, let
  $G(z) = \sum\limits_{k = 0}^\infty {{p_k}{z^k}} $. Then
  $\sum\limits_{k = 0}^\infty {{p_k} = 1} $ if and only if $\mathop
  {\lim }\limits_{z \nearrow 1} G(z) = 1$.
\end{lemma}
\begin{proof}
  Necessity: We suppose \eqref{eq:cont-thm-signed1} holds and define
  $G(z)$ by \eqref{eq:cont-thm-signed2}. If ${p_k^{(n)}}$ is a bounded
  and tight signed measure, then there exists an $M$ such that $\max
  \left\{ \left| p_{k} \right|,\left| {p_k^{(n)}} \right| \right\} \le
  M$ for $k = K+1,K+2, \ldots $. When $0 < z < 1$, it follows that
  \[\left| {{G_n}(z) - G(z)} \right| = \left| {\sum\limits_{k = 0}^\infty  {(p_k^{(n)} - {p_k})} {z^k}} \right| \le \left| {\sum\limits_{k = 0}^K {(p_k^{(n)} - {p_k})} } \right| + 2M\left| {\sum\limits_{k = K + 1}^\infty  {{z^k}} } \right|.\]
  Set $\varepsilon > 0$. Fix $K$ such that $\left| {\sum\limits_{k = K
        + 1}^\infty {{z^k}} } \right| = \frac{{{z^K}}}{{1 - z}} \le
  \varepsilon/2M $, and choose $N$ sufficiently large such that $\left|
    {\sum\limits_{k = 0}^K {(p_k^{(n)} - {p_k})} } \right| \le
  \varepsilon/2$ for all $n \ge N$. Then $\left| {{G_n}(z) - G(z)}
  \right| \le \varepsilon $ for $n \ge N$. Hence
  \eqref{eq:cont-thm-signed2} is true.

  Sufficiency: Assuming \eqref{eq:cont-thm-signed2} is true for
  $0 < z < 1$, then there is a subsequence $\{ {n^{(1)}}\} $ so that
  $\mathop {\lim }\limits_{{n^{(1)}} \to \infty } {G_{{n^{(1)}}}}(z) =
  G(z)$.
  Note that $\{ p_k^{({n^{(1)}})}\}$ is bounded and uniformly
  tight. From Prohorov's theorem for signed measures, we get a
  sub-subsequence $\{ {n^{(2)}}\} $ such that
  $\mathop {\lim }\limits_{{n^{(2)}} \to \infty } p_k^{({n^{(2)}})} =
  {p_k}$,
  thus
  $\mathop {\lim }\limits_{{n^{(2)}} \to \infty } {G_{{n^{(2)}}}}(z) =
  G(z)$.
  The coefficients of the signed p.g.f.~(identity theorem) do not
  depend on the choice of $\{ {n^{(1)}}\} $ and $\{ {n^{(2)}}\}$. Thus
  we have \eqref{eq:cont-thm-signed1}. If every convergent subsequence
  $p_k^{({n^{(1)}})}$ did not have the same limit ${p_k}$, then
  neither would $p_k^{({n^{(2)}})}$. This contradiction leads to the
  truth of \eqref{eq:cont-thm-signed1}.

  Moreover, the final result is deduced from following equalities:
  \[1 = \mathop {\lim }\limits_{z \nearrow 1} \mathop {\lim
  }\limits_{n \to \infty } {G_n}(z) = \mathop {\lim }\limits_{n \to
    \infty } \mathop {\lim }\limits_{z \nearrow 1} {G_n}(z) = \mathop
  {\lim }\limits_{n \to \infty } \mathop {\lim }\limits_{z \nearrow 1}
  \sum\limits_{k = 0}^\infty {p_k^{(n)}} {z^k} = \sum\limits_{k =
    0}^\infty {{p_k}},\] for $0 < z < 1$.
\end{proof}

\begin{lemma} \label{lem:quasi-ID-pgf-not-zero}
  Let $G(z)$ be a p.g.f. of signed discrete infinitely divisible r.v.. Then $G(z) \ne 0$
  for all $\left| z \right| \le 1$.
\end{lemma}
\begin{proof}
  Given $n \in \mathbb{N}$, if $X$ is signed discrete infinitely divisible, then
  $G(z) = {[{G_n}(z)]^{\frac{1}{n}}}$ where ${G_n}(z)$ is a signed
  p.g.f. Let
  \[ P(z) = \mathop {\lim }\limits_{n \to \infty } {G_n}(z) =
  \begin{cases}
    1, & \mbox{if $G(z) \ne 0$,}\\
    0, & \mbox{if $G(z) = 0$.}
  \end{cases} \] The function $G(z)$ is continuous in $z$ for $\left| z
  \right| \le 1$ and $G(1) = 1$, likewise so is $G_{n}(z) =
  [G(z)]^{\frac{1}{n}}$. By the continuity theorem for signed
  p.g.f.'s, $P(z)$ is a signed p.g.f.~and $P(1) = 1$. Since $P(z)$ is
  continuous for $\left| z \right| \le 1$, hence $P(z)=1$ for all
  $\left| z \right| \le 1$. The above statements show that ${G(z) \ne
    0}$ for every $\left| z \right| \le 1$.
\end{proof}

The continuity theorem for the p.g.f.'s ~of signed r.v.'s will be
applied in the derivation of the general form for discrete quasi ID
distributions. Putting all the above together we get a generalisation of the discrete
compound Poisson distribution. Now we state and prove our characterization of discrete
quasi infinitely divisible distributions.

\begin{theorem}[Characterization of signed discrete ID distributions] \label{thm:char-quasi-ID}
  A discrete distribution is signed discrete infinitely divisible if
  and only if it is a discrete pseudo compound Poisson distribution.
\end{theorem}
\begin{proof}
  Sufficiency: Given $n \in \mathbb{N}$ and $P_X(z)$, if $X$ is DPCP
  distributed then
  \[ [P_{X}(z)]^{\frac{1}{n}} = \exp \left\{ \sum \limits_{i =
      1}^{\infty} \frac{1}{n} \alpha_{i} \lambda (z^{i} - 1) \right\}
  \buildrel \Delta \over = \sum\limits_{k = 0}^\infty {p_k^{(n)}}
  {z^k}.\] By using the L{\'e}vy-Wiener theorem, we have
  $\sum\limits_{k = 0}^{\infty} \left| p_{k}^{(n)} \right| <
  \infty$. So $X$ is signed discrete ID distributed.

  Necessity: Lemma \ref{lem:quasi-ID-pgf-not-zero} and Lemma
  \ref{lem:char-DPCP} say respectively that the p.g.f.~of a signed discrete ID r.v.~$X$ has no zeros and any p.g.f.~that has no zeros is
  the p.g.f.~of a DPCP distribution. Consequently, $X$ is DPCP
  distributed.
\end{proof}

Similar to the criterion for discrete infinite divisibility from
Feller's characterization: a function $h$ is an infinitely divisible
p.g.f.~if and only if $h(1) = 1$ and
\[\ln h(z) - \ln h(0) = \sum\limits_{k = 1}^{\infty} a_{k}z^{k},\]
where $a_{k} \ge 0$ and $\sum\limits_{k = 1}^{\infty} a_{k} = \lambda <
\infty$ (see p.~290 of \cite{feller68}). The related open problem
was first studied by \cite{levy37}.

\begin{problem}
  If some $a_{i}$ are negative, under what necessary and sufficient
  conditions on $a_i$ is \[ \exp \left\{ \sum\limits_{i = 0}^{\infty}
    a_{i}(z^{i}-1)\right\} \] a p.g.f.?
\end{problem}

\cite{levy37} proved that $ P(z) = \exp \left\{ \sum \limits_{i =
    1}^{m} \alpha_{i} (z^{i} - 1) \right\} $ is not a p.g.f.~unless a
term with a sufficiently small negative coefficient is preceded by one
term with a positive coefficient and followed by at least two terms
with positive coefficients as well (see \cite{johnson05},
pp.~393--394), namely, the conditions are $a_{1} > 0$, $a_{m
  - 1} > 0$, and $a_{m} > 0$. \cite{milne93} considered the
multivariate form of (6) and gave some conditions under which the
exponential of a multivariate polynomial is a p.g.f. For $m = 4$, \cite{van78} gave four inequalities to ensure \[ A(z) =
\mathrm{e}^{a(z-1) - b({z^2} - 1) + c(z^{3} - 1) + d(z^{4} - 1)} \] is
a p.g.f.; the restrictions are $a,b,c,d > 0$ and $ b \le \min \{
\frac{a^{2}}{3}, \frac{c}{a}, \frac{ad}{2c},
\frac{c^{2}}{3d}\}$.

Next we give a few examples of DPCP distributions and the Bernoulli
distribution.

\begin{example}
  The p.g.f.~$P(z) = p + (1 - p)z$ on $\left| z \right| \le 1$ has the
  pseudo compound Poisson representation \[\ln [p + (1 - p)z] = \ln p
  - \ln p\sum\limits_{i = 1}^{\infty} \frac{( - 1)^{i - 1}}{ - i\ln p}
  \left(\frac{1 - p}{p} \right)^{i} z^{i} .\] Thus we have $a_{i} =
  \frac{( - 1)^{i - 1}}{ - i\ln p} \left(\frac{{1 - p}}{p}
  \right)^{i}$.
  \begin{enumerate}[label={\upshape(\roman*)}]
  \item If $p > 0.5$ then $X$ is DPCP distributed since $P(z)$ has no
    zeros.
  \item If $p = 0.5$ then $\sum _{i = 1}^\infty {a_i}$ is
    conditionally convergent and  $P(z)$ has the zero $z=-1$.
  \item If $p < 0.5$, then $\sum _{i = 1}^\infty {a_i}$ is divergent
    and $P(z)$ has zero $z{\text{ = }}\frac{p}{{p - 1}}$.
  \end{enumerate}
\end{example}

A more general example comes from the following corollary, see also \cite{zhang2014}.
\begin{corollary}
  For any discrete r.v.~$X$ with p.g.f.~$G(z)$, if $p_{0} > p_{1} >
  p_{2} > \cdots$, then $X$ is DPCP distributed..
\end{corollary}
\begin{proof}
  It can be shown that $G(z)$ has no zeros in $\left| z \right| < 1$,
  since $\left| {(1 - z)G(z)} \right| > 0$ for $\left| z \right| < 1$,
  see Theorem 2.3.1 of \cite{kerns04}, p.~29. And $z = \pm 1$ are not
  zeros due to the facts that $G(1) = 1$ and $G( - 1) = {p_0} - {p_1}
  + {p_2} - {p_3} + \cdots > 0$.
\end{proof}

The next corollary will be useful in fitting zero-inflated discrete distribution.

\begin{corollary} \label{cor:disc-greater-0.5}
  For any discrete r.v.~$X$, if $\Pr(X = 0) > 0.5$ then $X$ is DPCP
  distributed.
\end{corollary}
\begin{proof}
  If ${p_0} = \Pr(X = 0) > 0.5$, then $\left| {G(z)} \right| \geqslant {p_0} - \sum\limits_{i = 1}^\infty  {{p_i}}  =2{p_0}-1 > 0.$
\end{proof}

Based on the physical background of the fractal calculus,
\cite{laskin03} generalised ${\mathrm{e}^x}$ to Mittag-Leffler
functions:
\begin{equation}
{{\text{E}}_\mu }(x) = \sum\limits_{m = 0}^\infty  {\frac{{{x^m}}}{{\Gamma (\mu m + 1)}}} ,(0 < \mu  \leqslant {\text{1)}}
\end{equation}
Then, the fractional Poisson process $N_\lambda ^\nu (t)$ has p.g.f.
\[{{\text{E}}_\nu }[\lambda {t^\nu }(z - 1)],(\left| z \right| \leqslant 1,\lambda  > 0,0 < \nu  \leqslant 1).\]

The paper \cite{beghin14} extended the fractional Poisson process to
the discrete compound fractional Poisson process $M(t) =
\sum\limits_{k = 1}^{N_\lambda ^\nu (t)} {{X_k}}$. The p.g.f.~of
$M(t)$ is
\[\text{E}_{\nu} \left({t^\nu }\lambda \sum\limits_{i = 1}^\infty
  {{\alpha _i}} ({z^i} - 1) \right), \] where $\left| z \right|
\leqslant 1,{\alpha _i} \geqslant 0$, $0 < \nu \leqslant 1$, and ${\{
  {X_n}\} _{n \geqslant 1}}$ is a sequence of i.i.d.~discrete r.v.'s
independent of the fractional Poisson process $N_\lambda ^\nu (t)$.

The Proposition 3.1 of \cite{vellaisamy13} shows that the
one-dimensional distributions of the fractional Poisson process
$N_\lambda ^\nu (t)$, $0 < \nu < 1$, are not infinitely
divisible. Fortunately for us, it is quasi ID distributed because
Mittag-Leffler functions have no real zero as $0 < \nu < 1$. It
remains to use the following lemma, then Corollary 2.5 follows.
\begin{lemma}
  For the single parameter Mittag-Leffler function (10), if $0 < \nu
  \leqslant 1$, then $\text{E}_{\nu}(x)$ has no real zeros.
\end{lemma}
\begin{proof}
  It can be shown that $\text{E}_{1}(z) = \mathrm{e}^z$ has no zeros for all
  non-negative $z$. Just considering negative $z$ in the case $0 < \mu
  < 1$, the proof can be found in Theorem 4.1.1 of \cite{popov13}
  which states: ``The Mittag-Leffler function $\text{E}_{\rho} (z; a) =
  \sum\limits_{n = 0}^\infty {\frac{{{z^n}}}{{\Gamma (a + n/\rho
        )}}}$, where $\rho > 0,a \in \mathbb{C}$, and if $a \in
  \left(\bigcup\nolimits_{n = 0}^\infty {[ - n + \frac{1}{\rho }, -
      n{\text{ + }}1]} \right) \cup [1, + \infty )$, then $\text{E}_\rho(z;a
  )$ has no negative roots.''

\end{proof}

\textbf{Remark 2}: The other proof can be found in p.~453 of
\cite{feller71}. The Mittag-Leffler function $\text{E}_\mu (x)$ can be
written as a moment generating function
$\text{E}_{\mu}(t) = \E \mathrm{e}^{ tX} > 0$, where $X$ is the
transformation of a positive ${\alpha }$ stable distributed
r.v.~$Y = {X^{ - \frac{1}{\alpha }}}$ with moment generating function
$\E \mathrm{e}^{-tY} = {{\rm{e}}^{ - {t^\alpha }}}$.

\begin{corollary}
  The discrete compound fractional Poisson process $M(t)$ is DPCP
  distributed; so too is the fractional Poisson process.
\end{corollary}

As another special case, we have the following result. Another
generalization of Poisson distribution, the \emph{mixed Poisson}
distribution, is also DPCP.
\begin{corollary}
Let $X$ be a mixed Poisson r.v.\ with p.m.f.
\[P(X = n) = \int_0^{ + \infty } {\frac{{{\lambda ^n}}}{{n!}}} {e^{ - \lambda }}dF(\lambda ),(n = 0,1,2 \cdots ),\]
where $F(\lambda )$ is a distribution function, then $X$ is DPCP distributed.
\end{corollary}
\begin{proof}
  To prove this corollary we need to show that the p.g.f.\ of the
  mixed Poisson distribution has no zeros. Obviously,
\[G(z) = \int_0^{ + \infty } {{e^{\lambda (z - 1)}}} dF(\lambda ) > 0.\]
\end{proof}
Notice that if $F(\lambda )$ is an infinitely divisible distribution,
then $X$ is discrete compound Poisson distributed, see
\cite{maceda48}. This is the well-known Maceda's mixed Poisson with
infinitely divisible mixing law. The mixed Poisson distribution is widely applicable in the non-life insurance
science.

\section{Signed integer-valued ID}

In this section we define a class of signed integer-valued infinitely
divisible distributions which is wider than the class of
integer-valued infinitely divisible distributions. Next, we show that
a signed integer-valued infinitely divisible characteristic function never vanishes. Then we show that a
distribution is signed integer-valued infinitely divisible if and only
if it is an integer-valued pseudo compound Poisson distribution.

\begin{definition}[Integer-valued pseudo compound Poisson, IPCP]
  Let $X$ be an integer-valued random variable $X$ with $\Pr(X = k) =
  P_{k}$, $k \in \Bbb Z$. We say that $X$ has an integer-valued pseudo
  compound Poisson distribution if its characteristic function has the
  form
\begin{equation} \label{eq:def-IPCP}
\varphi (\theta) = \sum\limits_{k = -\infty }^{\infty}  P_{k} \mathrm{e}^{\mathrm{i}k\theta} = \exp \left\{ \sum \limits_{k =  -\infty}^{\infty} \alpha_{k} \lambda ( \mathrm{e}^{\mathrm{i}k\theta}  - 1)  \right\},
\end{equation}
where $\alpha_{0} = 0$, $\sum\limits_{k = - \infty }^\infty
\hspace{-0.5em} \alpha_{k} = 1$, $\sum\limits_{k = - \infty }^\infty
\hspace{-0.5em} {\left| {{\alpha _k}} \right|} < \infty$, $\alpha_{i}
\in \Bbb{R}$, and $\lambda > 0$.
\end{definition}

The early documental record of IPCP is in Paul L{\'e}vy 's monumental monograph of modern probability theory, see page 191 of \cite{levy37book}. Recent research about IPCP can be found in \cite{karymov05}.

As an example, consider $\varphi (\theta ) = \frac{1}{3} +
\frac{2}{3}{\mathrm{e}^{\mathrm{i}\theta }}$. We would like to write
$\varphi (\theta )$ as an exponential function. On our first attempt
we might try the Taylor series expansion for $\ln \left(\frac{1}{3} +
\frac{2}{3}z \right)$, but $\frac{1}{3} + \frac{2}{3}z$ vanishes at $z =
-1/2$, so this method cannot be employed. On our second attempt we
might look to the Fourier inversion formula for $\ln \left(\frac{1}{3}
  + \frac{2}{3} \mathrm{e}^{\mathrm{i}\theta}\right)$ since
$\frac{1}{3} + \frac{2}{3}\mathrm{e}^{\mathrm{i}\theta }$ has no
zeros. Let $f(\theta)$ be an integrable function on $[0,2\pi ]$. The
Fourier coefficients $c_{n}$ for $n \in {\Bbb Z}$ of $f(\theta )$ are
defined by
\[{c_n} = \frac{1}{{2\pi }}\int_0^{2\pi } {f(\theta)\, \mathrm{e}^{
    -\mathrm{i}n\theta } \, \mathrm{d}\theta},\quad n \in {\Bbb Z}.\]
In this example the Fourier coefficients of $\varphi (\theta )$ are
${c_n} = \frac{1}{{2\pi }}\int_0^{2\pi } {\ln \left( \frac{1}{3} +
    \frac{2}{3}{\mathrm{e}^{\mathrm{i}\theta }}\right){\mathrm{e}^{ -
      \mathrm{i}n\theta }}} \mathrm{d}\theta$, for $n \in {\Bbb Z}$,
and \[ \mathop {\lim }\limits_{M \to \infty } \sum\limits_{k = - M}^M
{{c_k}{\mathrm{e}^{\mathrm{i}k\theta }}} = f(\theta ) = \ln
\left(\frac{1}{3} + \frac{2}{3}{\mathrm{e}^{\mathrm{i}\theta
    }}\right).\]

\textbf{Remark 3} This example illustrates that a discrete r.v. may be signed integer-valued ID r.v. but not be signed discrete ID r.v.!

We enlisted $\mathrm{Maple}^{\circledR}$ to compute the $c_{n}$'s and
graphed them in Figure 1. Here, $c_{0} = \lambda$ and
${c_n} = \lambda \alpha_{n}$ for $n \ne 0$. The L{\'e}vy-Wiener
theorem guarantees that
$\sum\limits_{n = - \infty }^\infty \hspace{-0.75em}\left| \alpha_{n}
\right| < \infty$.

\begin{figure}[!ht]
\centering
\includegraphics[height=0.45\textwidth=1]{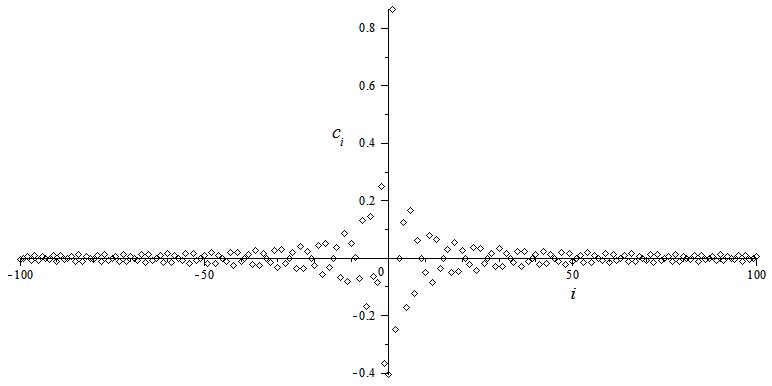}
\caption{Fourier coefficients for $\ln \left(\frac{1}{3} + \frac{2}{3}{\mathrm{e}^{\mathrm{i}\theta }}\right) = \sum c_{k} \mathrm{e}^{\mathrm{i}k\theta}$.}
\end{figure}

In the following, we
list three equivalent characterizations of an IPCP distributed
r.v.~with characteristic function
$\exp \left\{ {\sum\limits_{k = - \infty }^\infty {{\alpha _k}}
    \lambda ({{\rm{e}}^{{\rm{i}}k\theta }} - 1)} \right\}$
in \eqref{eq:def-IPCP}. The axiomatic derivations of IPCP
distributions can be obtained from any one of these characterizations.

\begin{corollary}

\begin{enumerate}[label={\upshape($\arabic*^{\circ}$)}]
\item {\upshape Signed compound Poisson:} Assume that $N$ is Poisson
  distributed with
  $\Pr(N = i) = \frac{\lambda^{i}}{i!}  \mathrm{e}^{-\lambda},(\lambda
  \in {\rm{R}})$,
  denoted by $N \sim Po(\lambda )$. Then, $X$ can be decomposed as
  \[X = Y_{1} + Y_{2} + \cdots + Y_{N},\] where the $Y_{i}$ are
  i.i.d.~integer-valued signed r.v.'s with signed probability density
  $\Pr( Y_{1} = k) = \alpha_{k}$, with $\sum\limits_{k = - \infty
  }^{\infty} \alpha_{k} = 1$, $\sum\limits_{k = - \infty }^{\infty}
  \left| \alpha_{k} \right| < \infty$ and $N$ independent of $Y_{i}$.

  \item {\upshape Sum of weighted signed Poisson:} $X$ can be decomposed as
    \[X = \sum \limits_{k = -\infty,\ k \ne 0}^{\infty} k N_{k}, \]
    where $N_{k}$ for $k \in \Bbb Z$ are independently signed Poisson
    distributed $N_{k} \sim Po(\alpha_{k} \lambda)$ with signed
    probability density $P({N_k} = n) = \frac{{{{(\lambda {\alpha _k})}^n}}}{{n!}}{e^{ - {\alpha _k}\lambda }}$, where $\alpha_{k}$ and $\lambda
    $ are defined in $1^{\circ}$.

  \item {\upshape Difference of discrete pseudo compound Poisson:} Let
    $\exp \left\{ \sum\limits_{i = 1}^\infty {\alpha _i^ + {\lambda _
          + }({z^i}} - 1) \right\}$ and $\exp \left\{ \sum\limits_{i =
        1}^\infty {\alpha _{ - i}^ - {\lambda _ - }({z^i}} - 1)
    \right\} $ be the p.g.f.'s of discrete r.v.'s $X_{+}$ and $X_{-}$,
    respectively. $\alpha_{k}$ and $\lambda
    $ are defined in $1^{\circ}$. Then $X$ can be seen as a difference of two
    independent r.v.'s \[ X = X_{+} - X_{-},\] where $\lambda_{+} = \sum
    \limits_{i = 1}^{\infty} \alpha_{i}$ and $\alpha_{i}^{+} =
    \alpha_{i}/\lambda_{+}$; also $\lambda_{-} = \sum \limits_{i =
      1}^{\infty}\alpha_{-i}$ and $\alpha _{ - i}^ - = \alpha
    _{-i}/\lambda_{-}$.

\end{enumerate}
\end{corollary}
\begin{proof}
  It is easy to check $(1^{\circ})$ -- $(3^{\circ})$ by examining the
  characteristic function.
\end{proof}

Next, we discuss the if-and-only-if relationship between the
signed integer-valued ID and integer-valued pseudo compound Poisson
distributions. This equivalence also holds for integer-valued
infinitely divisible and integer-valued compound Poisson distributions.

\begin{definition}
  A characteristic function (or the integer-valued r.v.~$X$) is said
  to be \emph{signed integer-valued infinitely divisible} if for every $n \in
  \mathbb{N}$, $\varphi_{X}(\theta)$ is the $n$-power of some
  characteristic function with signed probability density, namely,
  \[ \varphi_{X}(\theta) = [\varphi_{X_{n}} (\theta)]^{n} =
  \left(\sum \limits_{k = -\infty}^{\infty} p_{k}^{(n)} \mathrm{e}^{k\mathrm{i}\theta }\right)^{n},\] where \[\sum \limits_{k = -\infty}^{\infty} p_{k}^{(n)} =
  1, \quad \sum\limits_{k = -\infty }^{\infty} \left| p_{k}^{(n)} \right| <
  \infty,\quad p_{k}^{(n)} \in \Bbb{R}.\]
\end{definition}

To obtain our characterization for signed integer-valued ID
distributions we need Prohorov's theorem for signed measures, the
proof of which being found in \cite{bogachev07}. Applying Prohorov's
theorem, we obtain a continuity theorem for characteristic functions
with signed probability densities. For more reading on the application
of Prohorov's theorem to signed measures, see Theorem 2.2 and Theorem
2.3 of \cite{baez93}.

\begin{lemma}[Continuity theorem for signed characteristic functions,
  \cite{baez93}]
  (i) ${\mu _n} \Rightarrow \mu $ if and only if ${\hat \mu _n}
  \Rightarrow \hat \mu $ a.e.~and $\{ {\mu _n}\} $ is bounded and
  tight; (ii) Let ${\mu _n}$ be a complex measure and ${\hat \mu _n}$ be
  the characteristic function of ${\mu _n}$. If ${\hat \mu _n} \to g \
  a.e.$ and $\{ {\mu _n}\} $ is bounded and tight, then there exists a
  signed measure such that $\hat \mu = g$ and ${\mu _n} \Rightarrow
  g$.
\end{lemma}

\begin{lemma}\label{lem:quasi-ID-CF-not-zero}
  Let ${\varphi _X}(\theta )$ be a signed integer-valued infinitely divisible
  characteristic function. Then $\varphi_{X}(\theta) \ne 0$ for all
  $\theta$.
\end{lemma}
\begin{proof}
  Given $n \in \mathbb{N}$ and ${\varphi _X}(\theta )$, if $X$ is
  signed integer-valued infinitely divisible then $\varphi_{X}(\theta) = \left( \left|
      \varphi_{{X_n}} (\theta) \right|^{2} \right)^{\frac{1}{2n}}$
  where ${\varphi _{{X_n}}}(\theta )$ is a characteristic function
  with signed probability. Note that
  \[{\left| {{\varphi _{{X_n}}}(\theta )} \right|^2} = {\varphi
    _{{X_n}}}(\theta ){{\bar \varphi }_{{Y_n}}}(\theta ) = \E
  \mathrm{e}^{\mathrm{i}\theta {X_n}} \E
  {\mathrm{e}^{-\mathrm{i}\theta {Y_n}}} = \E \mathrm{e}^{\mathrm{i}
    \theta ({X_n} - {Y_n})},\] where ${X_n}\mathop = \limits^d {Y_n}$,
  is also a characteristic function with signed probability. Let
\[\psi (\theta ) = \mathop {\lim }\limits_{n \to \infty } {\left| {{\varphi _{{X_n}}}(\theta )} \right|^2} = \begin{cases}
  1, & \mbox{if $\varphi_{X}(\theta) \ne 0$,} \\
  0, & \mbox{if $\varphi_{X}(\theta) = 0$.}
\end{cases} \]
The function ${\varphi _X}(\theta )$ is continuous in $\theta$
for all $\theta $ and ${\varphi _X}(0) = 1$, so too is $\left|
  \varphi_{X_{n}} (\theta ) \right|^{2} =
[\varphi_{X}(\theta)]^{2n}$. By the continuity theorem for signed
characteristic functions, $\psi (\theta )$ is a signed characteristic
function and $\psi (0) = 1$. Since ${\varphi _X}(\theta )$ is
continuous for all $\theta $, so is $\psi (\theta )$. Hence, $\psi
(\theta) = 1$ for all $\theta $. The above statements show that
$\varphi_{X}(\theta ) \ne 0$ for every $\theta \in \Bbb{R}$.
\end{proof}

The continuity theorem for signed characteristic functions can be used
to deduce the general theorem for signed integer-valued ID.

\begin{theorem}[Characterization for signed integer-valued ID] \label{thm:char-iv-quasi-ID}
  An integer-valued distribution is signed integer-valued infinitely divisible if
  and only if it is an integer-valued pseudo compound Poisson
  distribution.
\end{theorem}
\begin{proof}
  Sufficiency: For every $n \in \mathbb{N}$, ${\varphi _X}(\theta )$,
  if $X$ is IPCP distributed, then \[{[{\varphi _X}(\theta
    )]^{\frac{1}{n}}} = \exp \left\{\sum\limits_{k = - \infty }^\infty
    {\frac{1}{n}{\alpha _k}\lambda ({\mathrm{e}^{k\mathrm{i}\theta }}} - 1) \right\}
  {\rm{ }} \buildrel \Delta \over = \sum\limits_{k = - \infty }^\infty
  {p_k^{(n)}} {\mathrm{e}^{k\mathrm{i}\theta }}.\] By the L{\'e}vy-Wiener theorem we
  have $\sum\limits_{k = - \infty }^\infty {\left| {p_k^{(n)}}
    \right|} < \infty $. Hence, $X$ is signed integer-valued ID.

  Necessity: Lemma \ref{lem:quasi-ID-CF-not-zero} says that the characteristic function of a
signed integer-valued ID r.v.~$X$ never vanishes. Judging from the
  L{\'e}vy-Wiener theorem, any characteristic function with no zeros
  is the characteristic function of an IPCP distribution. Therefore
  $X$ is IPCP distributed.
\end{proof}

\section{Some bizarre properties of DCP related to signed r.v.}

\cite{ruzsa83} proves a theorem related to signed random variables and
\cite{szekely05} gives the following result.
\begin{proposition}[Construction of signed random variables] \label{prop:construct-sign-rv}
  If $f \in {L^1}$ and $\smallint f\,\mathrm{d}\omega > 0$, then one
  can find a $g \in L^{1}_{+}$ with $g\not \equiv 0$ such that $f * g
  \in L{^1}_{+}$, where $L^{1}_{+}$ is the set of all non-negative
  integrable functions. Moreover, we can choose $g$ such that its
  Fourier transform is always positive.
\end{proposition}
\begin{proposition}[Fundamental theorem of negative
  probability] \label{prop:fund-thm-neg-prob} If a random variable $R$
  has a negative probability density, then there exist two other
  independent random variables $Y, Z$ with ordinary (not negative)
  distributions such that $R + Y = Z$ in distribution, where the
  operation $+$ is under `convolutional plus'. Thus $R$ can be seen as
  the `difference' of two ordinary random variables.
\end{proposition}
This fact gives a new explanation for why some $\alpha_{k}$ may be
negative by Theorem \ref{thm:char-quasi-ID}. Let
$\varphi_{R}(\theta)$, $\varphi_{Y}(\theta)$, and
$\varphi_{Z}(\theta)$ be the corresponding characteristic functions in
Proposition \ref{prop:fund-thm-neg-prob}.

Then we have $\varphi_{R}(\theta)\varphi_{Y}(\theta) =
\varphi_{Z}(\theta)$. If $R$ is integer-valued quasi ID, then
$\varphi_{R}(\theta) = [\varphi_{R}^{(n)}(\theta)]^{n}$ and
$\varphi_{R}^{(n)}(\theta) \not \equiv 0$, so we have
\[\varphi _R^{(n)}(\theta )\sum\limits_{k = 1}^\infty  {{b_{nk}}} {{\text{e}}^{{\text{i}}k\theta }} = \sum\limits_{k = 1}^\infty  {{a_{nk}}} {{\text{e}}^{{\text{i}}k\theta }} \mbox{ and } \sum\limits_{k = 1}^\infty  {{b_{nk}}} {{\text{e}}^{{\text{i}}k\theta }}\not  \equiv 0,\]
 from Proposition
\ref{prop:construct-sign-rv}.  Hence,
\[{\varphi _R}(\theta ) = {[\varphi _R^{(n)}(\theta )]^n} = \frac{{{{\left( {\sum\limits_{k = 1}^\infty  {{a_{nk}}} {{\text{e}}^{{\text{i}}k\theta }}} \right)}^n}}}{{{{\left( {\sum\limits_{k = 1}^\infty  {{b_{nk}}} {{\text{e}}^{{\text{i}}k\theta }}} \right)}^n}}} = \frac{{\exp \left( {\sum\limits_{k = 1}^\infty  {v_{nk}'} {{\text{e}}^{{\text{i}}k\theta }}} \right)}}{{\exp \left( {\sum\limits_{k = 1}^\infty  {v_{nk}''} {{\text{e}}^{{\text{i}}k\theta }}} \right)}} = \exp \left\{ {\sum\limits_{k = 1}^\infty  {(v_{nk}' - v_{nk}'')} {{\text{e}}^{{\text{i}}k\theta }}} \right\},\]
  where $v_{nk}'$ and
$v_{nk}''$ are non-negative. They satisfy $\sum\limits_{k = 1}^\infty  {\left|v_{nk}' - v_{nk}'' \right|}  < \infty $
because $\sum\limits_{k = 1}^\infty  {{a_{nk}}} {{\text{e}}^{{\text{i}}k\theta }}$ and $\sum\limits_{k = 1}^\infty  {{b_{nk}}} {{\text{e}}^{{\text{i}}k\theta }}$ both have
no zero, and the L{\'e}vy-Wiener theorem makes sense.

For those r.v.'s $X$ whose characteristic function can be written as
$\varphi (\theta ) = \mathrm{e}^{\eta (\theta)}$, the
\emph{J{\o}rgensen set} $\Lambda (X)$ of positive reals (see
\cite{jorgensen97}) is defined by
\[ \Lambda (X) = \{ \lambda > 0: \mathrm{e}^{\lambda \eta (\theta )}
\mbox{ is a characteristic function}\}. \]

\cite{albeverio98} deduces the following two Propositions.
\begin{proposition}
  $\Lambda (X)$ is a semigroup containing 1. Furthermore, $\Lambda
  (X)$ is closed in ${\text{(0,}}\infty )$.
\end{proposition}

\begin{proposition}
  Let $S \subset (0,\infty )$ be a closed semigroup with ${\text{1}}
  \in {\text{S}}$. Then,
  \begin{enumerate}[label={\upshape(\arabic*)}]
  \item Either $S = (0,\infty )$ or $S \subset [\lambda,
    \infty)$ for some $\lambda > 0$.
  \item If the interior of $S$ is non-empty, then $S \supset [\lambda,
    \infty)$ for some $\lambda$, $0 < \lambda < \infty$.
  \end{enumerate}
\end{proposition}

Here we show an example. Suppose there exists a negative $a_{k}$ in
\[ \varphi(\theta) = \exp \left\{ \lambda \sum \limits_{k =
    1}^{\infty} \alpha_{k} ( \mathrm{e}^{\mathrm{i}k\theta} - 1)
\right\}. \]
According to Corollary \ref{cor:neg-alpha} below, for $\lambda > 0$
sufficiently small, $\varphi (\theta )$ cannot be a characteristic
function. It is evident that ${[\varphi (\theta )]^\lambda }$ is still
a characteristic function since
$\lambda \in \mathbb{N}\setminus\{ 0\}$ implies
$\mathbb{N}\setminus\{ 0\} \subset \Lambda (X)$. Given the independent
convolution of a Bernoulli r.v.~and a Gamma distributed r.v.~$X + Y$,
where $X,Y$ is non-degenerate, \cite{nahla11} find the set
$\Lambda (X + Y)$. If $X$ and $Y$ are non-degenerate independent
r.v.'s which have respectively Bernoulli and negative binomial
distributions, \cite{letac02} find the set $\Lambda (X + Y)$.
\cite{nakamura13} gives a signed discrete infinitely divisible characteristic
function $\varphi (\theta)$ such that
${[\varphi (\theta )]^u},u \in \Bbb{R}$, is not a characteristic
function except for the non-negative integers. To find the
J{\o}rgensen set of a signed discrete ID characteristic function or a
quasi ID distribution defined by L{\'e}vy-Khinchine representation
with signed L{\'e}vy measure (\ref{eq:levy-khinchine}) is a research
problem.

Here we show that the
J{\o}rgensen set of signed discrete infinitely divisible Bernoulli r.v.'s $X$ is
$\mathbb{N}$.

\begin{corollary}
  Let ${G_X}(z) = p + (1 - p)z$ be the p.g.f., where $p > 0.5$. Then the
  J{\o}rgensen set is $\mathbb{N}$.
\end{corollary}
\begin{proof}
  Trivially, $\mathbb{N}$ belongs to the J{\o}rgensen set of $X$. It
  remains to show that $r \in (0,\infty ]\backslash \mathbb{N} $
  cannot be in the J{\o}rgensen set. From Taylor's formula, we have
  \[{[p + (1 - p)z]^r} = {p^r} \left[1 + \sum\limits_{i = 1}^\infty
    {r(r - 1) \cdots (r - i + 1) \cdot \frac{1}{{i!}}{{
          \left(\frac{1 - p}{{p}} \right)}^i}{z^i}} \right] \buildrel
  \Delta \over = \sum\limits_{i = 0}^\infty {{c_i}{z^i}}. \] For each
  $r < i - 1$, $i = 2,3,\dots$, we consider two cases: (i) if there
  exists an $i$ such that ${c_i} < 0$, then the result follows; (ii)
  if there exists an $i$ such that ${c_i} > 0$, note that
  \[{c_i} = r(r - 1) \cdots (r - i + 1) \cdot \frac{1}{{i!}}{
    \left(\frac{1 - p}{p} \right)^i} \mbox{ and } {c_{i + 1}} = r(r
  - 1) \cdots (r - i) \cdot \frac{1}{{(i + 1)!}}{ \left(\frac{1 - p}{p} \right)^{i + 1}}.\] Then ${c_{i + 1}} =
  \frac{{{c_i}}}{{(i + 1)(r - i + 1)}} \left(\frac{1 - p}{p} \right)
  < 0$ since ${r - i + 1 < 0}$.
\end{proof}

\begin{corollary} \label{cor:neg-alpha}
  If there exists a negative $\alpha_{k}$ in \eqref{eq:DPCP}, then
  ${\varphi _X}(\theta ) = \exp \left\{ {\sum\limits_{k = 1}^\infty  {{\alpha _k}} \lambda ({{\text{e}}^{{\text{i}}k\theta }} - 1)} \right\}$ is not a characteristic function for $\lambda > 0$
  sufficiently small.
\end{corollary}
\begin{proof}
  The p.m.f.~of $X$ is \eqref{eq:DPCP-explicit}. For $m \geqslant 1$, we have $\mathop {\lim
}\limits_{\lambda \to 0} \frac{{{P_m}}}{\lambda } = {\alpha _m} <
0$. This is in contradiction of the fact that $P_{m}/\lambda > 0$.
\end{proof}

\textbf{Remark 4}  This theorem shows that we should not replace $X$ by $X(t)$ for $t
  \in [0,t)$ (namely, replace $\lambda $ by $\lambda t$) in
(4) to have the DPCP process. The DPCP processes have a strange
property that the $t$ can only belong to the set $\Lambda (X)$ which have semigroup properties.

It is well-known that Poisson processes are characterized by
stationary and independent increments.
\begin{proposition}[Axioms for the Poisson process]
  \label{prop-ax-poisson-process} If a nonnegative integer-valued
  process $\{ X(t),t \ge 0\} $ satisfies the following conditions:
\begin{enumerate}[label={\upshape(\roman*)}]
  \item Initial condition $X(0) = 0$ and $0 < \Pr(X(t) > 0) < 1$.
  \item $X(t)$ has stationary increments.
  \item $X(t)$ has independent increments.
  \item Let discrete probability mass be given by ${P_i}(t) = \Pr(
    \left. {X(t) = i} \right|X(0) = 0)$, with the probability of $1$
    and $i \geqslant 2$ events taking place in $[{t},\Delta t + {t})$
    given by ${P_1}(\Delta t) = {\lambda }\Delta t + o(\Delta t)$ and
    $\sum\limits_{i = 2}^\infty {{P_i}(\Delta t)} = o(\Delta t)$,
    respectively.
\end{enumerate}
Then $\{X(t),\ t \ge 0\} $ is a Poisson process.
\end{proposition}

\begin{proposition}[Axioms for the discrete compound Poisson process]
  \label{prop-ax-dcpp}
  If a nonnegative integer-valued process $\{ X(t),t \ge 0\} $
  satisfies each of (i)--(iii) in Proposition
  \ref{prop-ax-poisson-process} and in addition satisfies
  \begin{enumerate}[label={\upshape(\roman*)}]
    \setcounter{enumi}{3}
  \item ${P_k}(\Delta t) = {\alpha _k}\lambda \Delta t + o(\Delta t)$
    where $\sum\limits_{k = 1}^\infty {{\alpha _k}} = 1$ and $0 \le
    {\alpha _k} \le 1,$
  \end{enumerate}
  then $\{ X(t),t \ge 0\} $ is a discrete compound Poisson process.
\end{proposition}

Proposition \ref{prop-ax-poisson-process} is called the \emph{Bateman
  theorem}, see Section 2.3 of \cite{haight67}. The
processes in Proposition \ref{prop-ax-dcpp} can be seen in Section 3.8
of \cite{haight67}, where they are known as \emph{stuttering Poisson
  processes}. For the case discrete r.v., \cite{janossy50} extended axioms for the Poisson process to Proposition \ref{prop-ax-dcpp}. We
may ask if there exist similar extensions of Proposition
\ref{prop-ax-dcpp} to DPCP
processes. As the contradiction in the proof of Corollary \ref{cor:neg-alpha} shows, if
some $a_{k}$ in \eqref{eq:feller-CP-pgf} are negative, this will yield
${P_k}(\Delta t) = {\alpha _k}\lambda \Delta t + o(\Delta t) <
0$.
However, ${P_k}(\Delta t)$ must be non-negative. So there are not DPCP processes $X(t)$ on $t \in [0, + \infty )$ as some $a_{k}$ in
\eqref{eq:feller-CP-pgf} are negative.

To avoid the above contradictions, we highlight a strange
property of DPCP processes on the semigroup $\Lambda (X)$.

The properties of discrete pseudo compound
  Poisson: The DPCP process $\{ X(t),t \in \Lambda (X)\}$ satisfies
  the following conditions:
\begin{enumerate}[label={\upshape(\roman*)}]
\item Initial condition $X({t_0}) = 0$, where $t_{0} = \inf \{ \lambda
  > 0: [\lambda ,\infty ) \subset \Lambda (X)\}$ and $0 < \Pr(X(t) > 0)
  < 1$.
\item $X(t)$ has stationary increments.
\item $X(t)$ has independent increments.
\end{enumerate}

\section{Conclusion and comment}
It was the great mathematician Leopold Kronecker who once said, ``God
made the integers; all else is the work of man.'' It is in the spirit
of that proverb that the present work deals with discrete(integer-valued) r.v..

Feller's characterization of discrete ID (namely, non-negative
integer-valued infinitely divisible distributions) says that a
distribution is discrete ID if and only if it is discrete compound
Poisson. Further, we define the discrete pseudo compound
Poisson (DPCP) distribution whose p.g.f.~has the exponential
polynomial form ${e^{P(z)}}$, where $P(z)$ may contain negative
coefficients (except coefficient $P(0)$). Using the definition of
generalized infinitely divisible, then, we have an extension of
Feller's characterization which is that a distribution is discrete
quasi-ID if and only if it is discrete pseudo compound Poisson, which
is made possible by Prohorov's theorem for bounded and tight signed
measures. This theorem could be applied to the continuity theorem for
p.g.f.'s with negative coefficients, or the continuity theorem for
characteristic functions with signed measure.

If
$\exp \left\{ \lambda \sum \limits_{k =
    1}^{\infty} \alpha_{k} ( \mathrm{e}^{\mathrm{i}k\theta} - 1)
\right\}$ is characteristic function, the parameter $\lambda$ can not tend to $0$ when some $\alpha_{k}$ take
negative values. This
property is related to a characteristic function's J{\o}rgensen set,
which is the set such that the positive real power of the
characteristic function maintains as a characteristic function. It is
easy to see that the J{\o}rgensen set of an infinitely divisible
characteristic function is ${{\text{R}}^ + }$. To find the
J{\o}rgensen set of a quasi-infinitely divisible characteristic
function is an open problem; there are only some special cases of
J{\o}rgensen sets in the literature, see \cite{nakamura13},
\cite{nahla11}, \cite{letac02}, \cite{albeverio98}.

\section{Acknowledgments}
The authors want to thank Prof. Sz\'{e}kely, G. J. for his helpful suggestion. This work was partly supported by the National Natural Science Foundation of China (No.11201165). After \cite{zhang16} was accepted in Feb 2014, the first edition manuscript of this work was originally written by studying some concepts and references in \cite{sato14}. More theoretical results of quasi-infinitely divisible distributions can be found in recent manuscript \cite{lindner17}.

\end{document}